\documentclass[ECP]{ejpecp} 

\SHORTTITLE{Normal approximation for the mean field classical $N$-vector models} 

\TITLE {Error bounds in normal approximation for the squared-length of total spin in the mean field classical $N$-vector models \thanks{This work was supported
by National Foundation for Science and Technology Development (NAFOSTED), grant
no. 101.03-2015.11.}}

\AUTHORS{%
  L\^{e} V\v{a}n Th\`{a}nh \footnote{Department of Mathematics, Vinh University, Nghe An 42118, Vietnam.
    \EMAIL{levt@vinhuni.edu.vn}}
  \and 
  Nguyen Ngoc Tu \footnote{Department of Applied Sciences, HCMC University of Technology and Education, Ho Chi Minh City, Vietnam.} \footnote{	
	Department of Mathematics and Computer Science, University of Science, Viet Nam National University Ho Chi Minh City, Ho Chi Minh City, Vietnam.
     \BEMAIL{tunn@hcmute.edu.vn}}}

\KEYWORDS{Stein's method, Kolmogorov distance, Wasserstein distance, mean-field model} 

\AMSSUBJ{60F05} 

\SUBMITTED{July 05, 2018} 
\ACCEPTED{February 16, 2019} 

\VOLUME{24}
\YEAR{2019}
\PAPERNUM{16}
\DOI{218}

\ABSTRACT{This paper gives the Kolmogorov and Wasserstein bounds in normal  approximation
for the  squared-length of total spin in the mean field classical $N$-vector models.
The Kolmogorov bound is new while the Wasserstein bound improves a result
obtained recently by Kirkpatrick  and Nawaz [Journal of Statistical Physics,  \textbf{165} (2016), no. 6, 1114--1140].
The proof is based on Stein's method for exchangeable pairs.}

\newtheorem{thm}{Theorem}[section]
\newtheorem{prop}[thm]{Proposition}
\newtheorem{lem}[thm]{Lemma}

\newtheorem{rem}[thm]{Remark}
\newtheorem{cor}[thm]{Corollary}

\newcommand{\thmref}[1]{Theorem~{\rm \ref{#1}}}
\newcommand{\lemref}[1]{Lemma~{\rm \ref{#1}}}
\newcommand{\corref}[1]{Corollary~{\rm \ref{#1}}}
\newcommand{\propref}[1]{Proposition~{\rm \ref{#1}}}

\begin{document}

\section{Introduction and main result}

Let $N\ge 2$ be an integer, and let ${\mathbb{S}}^{N-1}$ denote the unit sphere in ${\mathbb{R}}^N$.
In this paper, we consider the mean-field classical $N$-vector spin models,
where each spin $\sigma_i$ is in ${\mathbb{S}}^{N-1}$, at a complete graph vertex $i$ among $n$ vertices (\cite[Chapter 9]{FV}).
The state space is $\Omega_n=({\mathbb{S}}^{N-1})^n$ with product measure $P_n=\mu\times\dots\times \mu$,
where $\mu$ is the uniform
probability measure on ${\mathbb{S}}^{N-1}$.
In the absence of an external field, each spin configuration $\sigma=(\sigma_1,\dots, \sigma_n)$ in the
state space $\Omega_n$ has a Hamiltonian defined by
\[H_n(\sigma)=-\dfrac{1}{2n}\sum_{i=1}^n \sum_{j=1}^n \langle \sigma_i , \sigma_j \rangle,\]
where $\langle \cdot , \cdot \rangle$ is the inner product in ${\mathbb{R}}^N$.
Let $\beta>0$ be the inverse temperature. 
The Gibbs measure with Hamiltonian $H_n$ is the probability measure $P_{n,\beta}$ on $\Omega_n$
with density function:
\[dP_{n,\beta}(\sigma)=\dfrac{1}{Z_{n,\beta}}\exp\left( -\beta H_{n}(\sigma)\right)dP_n(\sigma),\]
where $Z_{n,\beta}$ is the partition function:
$Z_{n,\beta}=\int_{\Omega_n}\exp\left(-\beta H_n (\sigma) \right)dP_n(\sigma).$
This model is also called the mean field $O(N)$ model. It reduces to the $XY$ model, the Heisenberg model and the Toy model
when $N=2,3,4$, respectively (see, e.g., \cite[p. 412]{FV}). 

Before proceeding, we
introduce the following notations.
 Throughout this paper, $Z$ is a 
standard normal random variable, and $\Phi(z)$ is the 
probability distribution function of $Z$.
For a real-valued function $f$, we write $\|f\|=\sup_{x}|f(x)|$.
The symbol $C$ denotes a positive constant which depends only on the inverse temperature $\beta$, and its value
may be different for each appearance. For two random variables $X$
and $Y$, the Wasserstein distance $d_{\mathrm{W}}$ and the Kolmogorov distance $d_{\mathrm{K}}$ between ${\mathcal{L}}(X)$ and ${\mathcal{L}}(Y)$ are as follows:
\[d_{\mathrm{W}}({\mathcal{L}}(X),{\mathcal{L}}(Y))=\sup_{\|h'\|\le 1}\left|Eh(X) -Eh(Y)\right|,\]
and
\[d_{\mathrm{K}}({\mathcal{L}}(X),{\mathcal{L}}(Y))= \sup_{z\in \mathbb{R}}\left|P\left(X\le z\right) -P\left(Y\le z\right)\right|.\]

In the Heisenberg model ($N=3$), Kirkpatrick and Meckes \cite{KM} established large deviation,
normal approximation results for total spin
$S_n=\sum_{i=1}^n \sigma_i$ in the
non-critical phase ($\beta\not=3$), 
 and a non-normal approximation result in the 
critical phase
($\beta=3$). The results in \cite{KM} are generalized by Kirkpatrick and Nawaz \cite{KN} 
to the mean field $N$-vector models with $N\ge 2$.

Let $I_{\nu}$ denote the modified Bessel function of the first kind
(see, e.g., \cite[p. 713]{AW}) 
and 
\begin{equation}\label{eq:add00Bessel}
f(x)=\dfrac{I_{\frac{N}{2}}(x)}{I_{\frac{N}{2}-1}(x)},\ x>0.
\end{equation}
By \lemref{le:boundf} in the Appendix, we have
\begin{equation}\label{eq:add01}
\left(\dfrac{f(x)}{x}\right)'<0 \mbox{ for all } x>0.
\end{equation}
We also have 
\begin{equation}\label{eq:add02}
\lim_{x\to 0^+}\dfrac{f(x)}{x}=\dfrac{1}{N} \mbox{ and } \lim_{x\to \infty}\dfrac{f(x)}{x}=0.
\end{equation}
In the case $\beta>N$, from \eqref{eq:add01} and \eqref{eq:add02}, there is a unique strictly positive solution $b$ to
the equation
\begin{align} \label{eq:beta}
x - \beta f(x)= 0.
\end{align}

Based on their large deviations, Kirkpatrick and Nawaz \cite{KN} 
argued 
that in the case $\beta>N$,
there exists $\varepsilon>0$ such that
\[P\left( \left| \dfrac{\beta|S_n|}{n} - b\right| \ge x \right) \le e^{-Cn {x}^2}\]
for all $0\le x\le \varepsilon$, where $S_n=\sum_{i=1}^n \sigma_i$ is total spin. It means that $|S_n|$ is 
close to $bn/\beta$ with high probability. On the other hand, 
all points on the hypersphere of radius $bn/\beta$
will have equal probability due to symmetry. 
Based on these facts, they considered
the fluctuations of the squared-length of total spin:
\begin{align} \label{eq:Wn}
W_n:= \sqrt{n}\left(\dfrac{\beta^2}{n^2 b^2} \Big|S_n \Big|^2 -1 \right),
\end{align}
where $S_n=\sum_{j=1}^n \sigma_j$. Let 
\begin{align} \label{eq:B2}
B^2=\dfrac{4\beta^2}{(1-\beta f^{'}(b))b^2}\left[1-\dfrac{(N-1)f(b)}{b}-(f(b))^2 \right].
\end{align}
Kirkpatrick and Nawaz \cite{KN} proved
that when $\beta >N$,
the bounded-Lipschitz distance between $W_n/B$ and $Z$ is bounded by $C (\log n/n)^{1/4}$.
Their proof is based on Stein's method for exchangeable pairs
(see Stein \cite{Stein86}). 
Recall that a random vector $(W,W')$ is called an exchangeable pair if $(W,W^{'})$ and $(W^{'}, W)$ have the same distribution.
Kirkpatrick and Nawaz \cite{KN} construct an exchangeable pair as follows.
Let $W_n$ be as in \eqref{eq:Wn} and let $\sigma' = \{\sigma_1',\ldots,\sigma_n'\}$, 
where for each $i$ fixed, $\sigma_i'$ is an independent copy of $\sigma_i$ given $\{\sigma_j, j\ne i\}$, i.e., given $\{\sigma_j, j\ne i\}$,
$\sigma_i'$ and $\sigma_i$ have the same distribution and $\sigma_i'$ is conditionally independent of $\sigma_i$
(see, e.g., \cite[p. 964]{EL}).
Let $I$ be a random index independent of all others and uniformly distributed over $\{1,\ldots,n\}$, and let
\begin{align} \label{eq:Wcom}
W_{n}^{'} = \sqrt{n}\left(\dfrac{\beta^2}{n^2 b^2}|S_n'|^2 - 1 \right),
\end{align}
where  $S_n' = \sum_{j=1}^n\sigma_j-\sigma_I + \sigma_I'$. 
Then $(W_n,W_{n}^{'})$ is an exchangeable pair (see Kirkpatrick and Nawaz \cite[p. 1124]{KN}, Kirkpatrick and Meckes \cite[p. 66]{KM}). 

The bound $C(\log n/n)^{1/4}$ obtained by Kirkpatrick and Nawaz \cite{KN}
is not sharp. The aim of this paper is to give
the Kolmogorov and Wasserstein distances between $W_n/B$ and $Z$ with optimal rate $C n^{-1/2}$.

The main result is the following theorem. We recall that, throughout this paper, $C$ is a positive constant which depends only on $\beta$, and its value
may be different for each appearance.
\begin{thm} \label{th:app2}
Let $\beta >N$ and $f$ be as in \eqref{eq:add00Bessel}. Let $b$ be the unique strictly positive solution to
the equation $x - \beta f(x)= 0$ and $B^2$ as in \eqref{eq:B2}.
For $W_n$ as defined in \eqref{eq:Wn}, we have
\begin{equation}\label{eq:W} \sup_{\|h'\|\le 1}\left|Eh(W_n/B) -Eh(Z)\right| \le C n^{-1/2}, \end{equation}
and
\begin{equation} \label{eq:K} \sup_{z\in \mathbb{R}}\left|P\left(W_n/B\le z\right) -\Phi(z)\right| \le C n^{-1/2}.\end{equation}
\end{thm}

The Wasserstein bound in Theorem \ref{th:app2} will be a consequence of the following proposition,
a version of Stein's method for exchangeable pairs.
It is a special case of
Theorem 2.4 of Eichelsbacher and L\"{o}we \cite{EL} or Theorem 13.1 in \cite{CGS}.
\begin{prop} \label{pr:EL}
Let $(W,W^{'})$ be an exchangeable pair and $\Delta=W-W'$.
If  $E(\Delta|W) = \lambda(W + R)$
for some random variable $R$ and $0 < \lambda <1$, then
\[ \sup_{\|h'\|\le 1}\left|Eh(W) -Eh(Z)\right| \le \sqrt{2/\pi}E\left| 1- \dfrac{1}{2\lambda}E(\Delta^2|W)\right|+\dfrac{1}{2\lambda}E|\Delta|^3+2E|R|.\]
\end{prop}

The Kolmogorov distance 
is more commonly used in probability and statistics, 
and is usually more difficult to handle than the Wasserstein distance.
Recently, Shao and Zhang \cite{SZ}  proved a very general theorem. Their result is as follows.
\begin{prop} \label{pr:ShaoZhang}
Let $(W,W^{'})$ be an exchangeable pair and $\Delta=W-W'$. 
Let $\Delta^*:=\Delta^*(W,W')$ be any random variable satisfying $\Delta^*(W,W')=\Delta^*(W',W)$ and $\Delta^*\ge |\Delta|$.
If
$E(\Delta|W) = \lambda(W + R)$
for some random variable $R$ and $0 < \lambda <1$, then
\[\sup_{z\in\mathbb{R}}\left|P(W\le z) - \Phi(z)\right| \le E\left| 1- \dfrac{1}{2\lambda}E(\Delta^2|W)\right|+\dfrac{1}{\lambda}E\left|E(\Delta \Delta^*|W)\right|+E|R|.\]
\end{prop}
Shao and Zhang \cite{SZ} applied their bound in \propref{pr:ShaoZhang} to get optimal bound
in many problems, including a bound of $O(n^{-1/2})$ for the Kolmogorov distance in 
normal approximation of total spin in the Heisenberg model. 
We note that
if $|\Delta|\le a$, then the following result is an immediate corollary
of \propref{pr:ShaoZhang}. In this case, the bound is much simpler than that of \propref{pr:ShaoZhang}.
\begin{cor} \label{co:cor01}
If $|\Delta|\le a$, then
\begin{equation}\label{eq:corSZ01}
\sup_{z\in\mathbb{R}}\left|P(W\le z) - \Phi(z)\right| \le E\left| 1- \dfrac{1}{2\lambda}E(\Delta^2|W)\right|
+(E|W|+1)a+E|R|.
\end{equation}
\end{cor}
\begin{proof}
In \propref{pr:ShaoZhang}, let $\Delta^*=a$, then
\begin{equation}\label{eq:corSZ02}
E\left|E(\Delta \Delta^*|W)\right|=aE\left|E(\Delta|W)\right|\le a\lambda (E|W|+E|R|).
\end{equation}
If $E|R|\ge 1$, then \eqref{eq:corSZ01} is trivial. If $E|R|<1$, then \eqref{eq:corSZ01} follows
immediately from \eqref{eq:corSZ02} and \propref{pr:ShaoZhang}.
\end{proof}
For $S_n=\sum_{i=1}^n \sigma_i$, and for $W_n$ and $W_{n}^{'}$ respectively defined in \eqref{eq:Wn} and \eqref{eq:Wcom}, 
we have
\[|\Delta|=|W_n-W_{n}^{'}|=\dfrac{\beta^2}{n^{3/2} b^2}\left||S_n|^2-|S_{n}^{'}|^2\right|\le \dfrac{4\beta^2}{n^{1/2} b^2},\]
since $|S_n|+|S_{n}^{'}|\le 2n$ and $|S_n|-|S_{n}^{'}|\le |\sigma_I-\sigma_{I}^{'}|\le 2.$
Therefore, we will apply \corref{co:cor01} to obtain the Kolmogorov bound in \thmref{th:app2}.

\section{Proof of the main result}

The proof of \thmref{th:app2} depends on Kirkpatrick and Nawaz's finding \cite{KN}.
Applying \propref{pr:EL} and \corref{co:cor01}, Theorem \ref{th:app2} follows from the following proposition.
\begin{prop} \label{pr:prop}
Let $\beta >N$, and let $f$ be as in \eqref{eq:add00Bessel},  $b$ the unique strictly positive solution to
the equation $x - \beta f(x)= 0$. Let $W_n$ and $W_{n}^{'}$ be as in \eqref{eq:Wn} and \eqref{eq:Wcom}, respectively.
Then the following statements hold:
\begin{description}
\item{(i)} $|W_{n}-W_{n}^{'}|\le 4\beta^2 b^{-2} n^{-1/2}$ and $EW_{n}^2\le C,$

\item{(ii)} $
E(W_n-W_{n}^{'}|W_n) = \lambda (W_n + R),$
where $\lambda = \dfrac{1-\beta f'(b)}{n}$ and $R$ is a random variable satisfying  $E|R| \le C  n^{-1/2}$,
\item{(iii)} $E\Big|\dfrac{1}{2\lambda}E((W_n-W_{n}^{'})^2|W_n))-B^2 \Big| \le C n^{-1/2},$
where $B^2$ is defined in \eqref{eq:B2}. 
\end{description}
 \end{prop}
\begin{rem} \label{re:diff}
 {\rm 
Kirkpatrick and Nawaz's \cite{KN} used their large deviation result for total spin $S_n$ to prove that $EW_{n}^2\le C \log n.$
Intuitively, we see that this bound
would be improved to $EW_{n}^2\le C$
since $W_n$ approximates a normal distribution.
By a more careful estimate, we can prove that
$E\left(\beta |S_n|/n-b\right)^2\le C/n$ (see \lemref{le:lema1}).
This will lead to desired bound $EW_{n}^2\le C$.
Kirkpatrick and Nawaz's \cite{KN} also
proved that 
\[E\Big|\dfrac{1}{2\lambda}E((W_n-W_{n}^{'})^2|W_n))-B^2 \Big| \le C \left(\dfrac{\log n}{n}\right)^{1/4}.\]
To get optimal bound of order $n^{-1/2}$ for this term, 
we use a fine estimate of function 
$f(x)=I_{\frac{N}{2}}(x)/I_{\frac{N}{2}-1}(x)$ (\lemref{le:boundf}) and a technique developed recently by Shao and Zhang \cite[Proof of (5.51)]{SZ}.
}
 \end{rem}

\begin{proof}[Proof of \propref{pr:prop}]
(i) We have
\begin{align*}
\begin{split}
|W_{n}-W_{n}^{'}|&=\dfrac{\beta^2}{b^2 n^{3/2}}\left| |S_{n}|^2-|S_{n}^{'}|^2\right|=\dfrac{\beta^2}{b^2 n^{3/2}}\left|\langle S_{n}+S_{n}^{'},S_{n}-S_{n}^{'}\rangle \right|\\
&\le \dfrac{2\beta^2n|S_{n}-S_{n}^{'}|}{b^2 n^{3/2}}= \dfrac{2\beta^2|\sigma_I - \sigma_I'|}{b^2 n^{1/2}}\le \dfrac{4\beta^2}{b^2 n^{1/2}}.
\end{split}
\end{align*}
The proof of the first half of (i) is completed. Now, apply \lemref{le:lema1} given in the Appendix, we have
\begin{align*}
\begin{split}
EW_{n}^2&=nE\left(\left(\dfrac{\beta|S_n|}{nb}+1\right)\left(\dfrac{\beta|S_n|}{nb}-1\right)\right)^2\le CnE\left(\dfrac{\beta|S_n|}{nb}-1\right)^2\le C.\\
\end{split}
\end{align*}

(ii) Kirkpatrick and Nawaz \cite[ equation (9)]{KN} showed that
\begin{equation}\label{eq:bound01}
E(W_n-W_{n}^{'}|W_n) = \dfrac{2}{n}W_n +\dfrac{2}{\sqrt{n}} - \dfrac{2\beta}{n^{1/2}b^2}\left(\dfrac{\beta|S_n|}{n}\right)f\left(\dfrac{\beta|S_n|}{n}\right) + R_1,
\end{equation}
where $R_1$ is a random variable satisying $E|R_1| \le C n^{-3/2}$. Set $g(x)=xf(x),x>0$. By Taylor's expansion, we have for some
positive random variable $\xi$:
\begin{equation}\label{eq:bound02}
g\left(\dfrac{\beta|S_n|}{n}\right) =
	g(b) + g'(b)\left(\dfrac{\beta|S_n|}{n}-b\right) +\dfrac{g''(\xi)}{2} \left(\dfrac{\beta|S_n|}{n}-b\right)^2.
\end{equation}
Set 
$V=\dfrac{\beta|S_n|}{nb}+1,$
we have $1\le V\le C$ and
\begin{equation}\label{eq:bound03}
\begin{split}
\dfrac{\beta|S_n|}{n}-b&=b\left(\dfrac{\beta|S_n|}{nb}-1\right)=\dfrac{bW_n}{\sqrt{n}V}
=\dfrac{bW_n}{2\sqrt{n}}-\dfrac{bW_n}{\sqrt{n}}\left(\dfrac{1}{2}-\dfrac{1}{V}\right)\\
&=\dfrac{bW_n}{2\sqrt{n}}-\dfrac{bW_n}{2\sqrt{n}V}\left(\dfrac{\beta|S_n|}{nb}-1\right)
=\dfrac{bW_n}{2\sqrt{n}}-\dfrac{bW_{n}^2}{2nV^2}.
\end{split}
\end{equation}
Combining \eqref{eq:bound01}-\eqref{eq:bound03} and noting that $b=\beta f(b)$, we have
\begin{align*}
&E(W_n-W_{n}^{'}|W_n)\\
 &= \dfrac{2W_n}{n} +\dfrac{2}{\sqrt{n}} +R_1 - \dfrac{2\beta}{n^{1/2}b^2} \left(g(b) 
+ g'(b)\left( \dfrac{bW_n}{2\sqrt{n}}-\dfrac{bW_{n}^2}{2nV^2}\right)+\dfrac{g''(\xi)}{2} \left(\dfrac{\beta|S_n|}{n}-b\right)^2\right) \\
&= \dfrac{2W_n}{n} +\dfrac{2}{\sqrt{n}} +R_1-\dfrac{2\beta}{n^{1/2}b^2} \left(\dfrac{b^2}{\beta} + \left(\dfrac{b}{\beta}+bf'(b) \right)\left( \dfrac{bW_n}{2\sqrt{n}}-\dfrac{bW_{n}^2}{2nV^2}\right)+\dfrac{g''(\xi)b^2W_{n}^2}{2nV^2}\right) \\
&= \dfrac{1-\beta f'(b)}{n}(W_n+R),
\end{align*}
where 
\[R=\dfrac{n}{1-\beta f'(b)}\left(R_1+\dfrac{\beta W_{n}^2}{n^{3/2 }V^2} \left(\dfrac{1}{\beta}+f'(b) -g''(\xi)\right)\right).\]
By \lemref{le:boundf} (ii), we have $|g''(\xi)|<6$.
Since $V\ge 1$, $EW_{n}^2\le C$ and $E|R_1|\le Cn^{-3/2}$, we conclude that $E|R|\le Cn^{-1/2}$.
The proof of (ii) is completed.

(iii) Denote $Id$ is the $n\times n$ identity matrix and set
$\sigma^{(i)}=S_n-\sigma_i,\ b_i = \beta |\sigma^{(i)}|/n,\ r_i=\dfrac{\sigma^{(i)}}{|\sigma^{(i)}|}.$
From Kirkpatrick and Nawaz \cite[Equations (11) and (12)]{KN}, we have
\begin{align*}
E((W_n-W_{n}^{'})^2|\sigma)&= 2\lambda B^2 + \dfrac{4\beta^4}{n^4 b^4}  \sum_{i=1}^n\left(1-\dfrac{N-1}{\beta}\right)\left(|\sigma^{(i)}|^2 - \dfrac{(n-1)^2 b^2}{\beta^2} \right) \\
&\quad  - \dfrac{8\beta^3}{n^4b^3} \sum_{i=1}^n \left(|\sigma^{(i)}|\langle \sigma_i,\sigma^{(i)}\rangle- \dfrac{n^2 b^3}{\beta^3}\right)\\
&\quad + \dfrac{4\beta^4}{n^4 b^4} \sum_{i=1}^n \left(\langle \sigma_i,\sigma^{(i)}\rangle^2  - \left(1-\dfrac{N-1}{\beta}\right)\dfrac{(n-1)^2 b^2}{\beta^2}\right)\\
&\quad + \dfrac{4\beta^4}{n^4 b^4} \sum_{i=1}^n\sum_{j,k\ne i}\sigma_j^T R_i'\sigma_k,
\end{align*}
where 
\[R_i'= \left(\dfrac{f(b_i)}{b_i}- \dfrac{1}{\beta} \right)Id - \left(\dfrac{Nf(b_i)}{b_i}- \dfrac{N}{\beta} \right)P_i - \left(f(b_i)- \dfrac{b}{\beta} \right)(r_i\sigma_i^T+\sigma_i r_i^T),\]
and $P_i$ is orthogonal projection onto $r_i$. 
Therefore,
\begin{align} \label{eq:BB} 
\frac{1}{2\lambda}E((W_n-W_{n}^{'})^2|\sigma))-B^2
= \dfrac{2\beta^4}{n^3 b^4(1-\beta f'(b))} \left(R_2-\dfrac{2b}{\beta} R_3+R_4+R_5\right),
\end{align}
where 
\begin{align*}
& R_2 = \sum_{i=1}^n\left(1-\dfrac{N-1}{\beta}\right)\left(|\sigma^{(i)}|^2 - \dfrac{(n-1)^2 b^2}{\beta^2} \right), \\
& R_3 = \sum_{i=1}^n \left(|\sigma^{(i)}|\langle \sigma_i,\sigma^{(i)}\rangle - \dfrac{n^2 b^3}{\beta^3}\right),\\
& R_4 = \sum_{i=1}^n \left(\langle \sigma_i,\sigma^{(i)}\rangle^2 - \left(1-\dfrac{N-1}{\beta}\right)\dfrac{(n-1)^2 b^2}{\beta^2}\right), \\
& R_5 =  \sum_{i=1}^n\sum_{j,k\ne i}\sigma_j^T R_i'\sigma_k.
\end{align*}
For $R_2$, noting that $|\sigma^{(i)}-S_n| \le 1$, then by \lemref{le:lema1}, we have
\begin{align}\label{eq:R2a}
\left(E\Big|\dfrac{\beta|\sigma^{(i)}|}{n} - b\Big|\right)^2 \le E\Big|\dfrac{\beta|\sigma^{(i)}|}{n} - b\Big|^2 
\le E\Big|\dfrac{\beta|S_n|}{n} - b\Big|^2 +\dfrac{C}{n^2}
\le \dfrac{C}{n}.
\end{align}
Thus, 
\begin{align} \label{eq:R2final}
\begin{split}
E|R_2|  &\le C \sum_{i=1}^n E\left||\sigma^{(i)}|^2 - \dfrac{(n-1)^2b^2}{\beta^2} \right|  \\
	&\le C n^2 \sum_{i=1}^n\left( E\left|\dfrac{\beta^2|\sigma^{(i)}|^2}{n^2} - b^2 \right|+\dfrac{(2n-1)b^2}{n^2}\right) \\
	&\le C n^2 \left(\sum_{i=1}^n E\left|\dfrac{\beta|\sigma^{(i)}|}{n} - b \right| +C\right)\le Cn^{5/2}.
\end{split}
\end{align}
For $R_3$, 
we have
\begin{align}\label{eq:R3final}
\begin{split}
E|R_3| &=  E\left| \sum_{i=1}^n \left(|S_n|\langle \sigma_i,S_n\rangle - \dfrac{n^2 b^3}{\beta^3}
+|\sigma^{(i)}|\langle \sigma_i,\sigma^{(i)}\rangle-|S_n|\langle \sigma_i,S_n\rangle\right)\right|\\ 
 &\le E\left| |S_n|^3 - \dfrac{n^3 b^3}{\beta^3}\right|+E\left|\sum_{i=1}^n |\sigma^{(i)}|\langle \sigma_i,\sigma^{(i)}\rangle-|S_n|\langle \sigma_i,S_n\rangle\right|\\ 
 &\le Cn^2E\left| |S_n| - \dfrac{nb}{\beta}\right|+E\left|\sum_{i=1}^n \left(|\sigma^{(i)}|-|S_n|\right)\langle \sigma_i,\sigma^{(i)}\rangle-|S_n|\langle \sigma_i,\sigma_i\rangle\right|\\ 
 &\le Cn^3E\left| \dfrac{\beta|S_n|}{n}-b\right|+E\sum_{i=1}^n \left(|\langle \sigma_i,\sigma^{(i)}\rangle|+|S_n|\right)\\ 
 &\le Cn^3 E\left| \dfrac{\beta|S_n|}{n}-b\right|+C n^2\le C n^{5/2}.
\end{split}
\end{align}
To bound $E|R_5|$, we note that
\begin{align*}
&\sum_{i=1}^n\sum_{j,k\ne i}\sigma_j^T R_i'\sigma_k\\
& = \sum_{i=1}^n\sum_{j,k\ne i} \left[ \left(\dfrac{f(b_i)}{b_i}- \dfrac{1}{\beta} \right)\langle \sigma_j, \sigma_k\rangle  - \left(f(b_i)- \dfrac{b}{\beta} \right)\sigma_j^T(r_i\sigma_i^T+\sigma_i r_i^T)\sigma_k\right]\\
&\quad -\sum_{i=1}^n\sum_{j,k\ne i} \left(\dfrac{Nf(b_i)}{b_i}- \dfrac{N}{\beta} \right)\sigma_j^T P_i\sigma_k\\
&= \sum_{i=1}^n\left[\left(\dfrac{f(b_i)}{b_i}- \dfrac{1}{\beta} \right)|\sigma^{(i)}|^2-2\left(f(b_i)- \dfrac{b}{\beta} \right) |\sigma^{(i)}|\langle \sigma^{(i)},\sigma_i\rangle\right]\\
&\quad -\sum_{i=1}^n\left(\dfrac{Nf(b_i)}{b_i}- \dfrac{N}{\beta} \right)\sum_{j,k\ne i} \text{Trace}(\sigma_k\sigma_j^T r_i r_i^T)\\
&= \sum_{i=1}^n\left[\left(\dfrac{f(b_i)}{b_i}- \dfrac{1}{\beta} \right)|\sigma^{(i)}|^2-2\left(f(b_i)- \dfrac{b}{\beta} \right) |\sigma^{(i)}|\langle \sigma^{(i)},\sigma_i\rangle\right]\\
&\quad -\sum_{i=1}^n \left(\dfrac{Nf(b_i)}{b_i}- \dfrac{N}{\beta} \right)\langle \sigma^{(i)},r_i \rangle^2 \\
&= \sum_{i=1}^n(1-N)\left(\dfrac{f(b_i)}{b_i}- \dfrac{1}{\beta} \right)|\sigma^{(i)}|^2-2\sum_{i=1}^n\left(f(b_i)- \dfrac{b}{\beta} \right) |\sigma^{(i)}|\langle \sigma^{(i)},\sigma_i\rangle\\
&:=R_{51}-2R_{52}.
\end{align*}
Since $1/\beta=f(b)/b$ and $b_i=\beta |\sigma^{(i)}|/n$, we have
\begin{align}\label{eq:R51}
\begin{split}
E|R_{51}|&=E\left|\sum_{i=1}^n(1-N)\left(\dfrac{f(b_i)}{b_i}- \dfrac{f(b)}{b} \right)|\sigma^{(i)}|^2\right|\\
& \le Cn^2 \sum_{i=1}^n E|b_i-b| \mbox{ (by \lemref{le:boundf} (iii) and the fact that $|\sigma^{(i)}|\le n$)}\\
& \le Cn^2 \sum_{i=1}^n E\left(\left|\dfrac{\beta |S_n|}{n}-b\right|+\dfrac{\beta}{n}\left(|\sigma^{(i)}|-|S_n|\right)\right)\\
&\le Cn^{5/2}\mbox{ (by \eqref{eq:R2a} and the fact that $||\sigma^{(i)}|-|S_n||\le 1$)}.
\end{split}
\end{align}
Similarly,
\begin{align}\label{eq:R52}
\begin{split}
E|R_{52}|&=E\left|\sum_{i=1}^n(1-N)\left(f(b_i)- f(b) \right)|\sigma^{(i)}|\langle \sigma^{(i)},\sigma_i\rangle\right|\\
& \le Cn^2 \sum_{i=1}^n E|b_i-b| \mbox{ (by \lemref{le:boundf} (i) and the fact that $|\sigma^{(i)}|\le n$)}\\
&\le Cn^{5/2}.
\end{split}
\end{align}
Combining \eqref{eq:R51} and \eqref{eq:R52}, we have
\begin{align}\label{eq:R5final}
\begin{split}
E|R_{5}|\le Cn^{5/2}.
\end{split}
\end{align}

Bounding $E|R_4|$ is the most difficult part.
Here we follow a technique developed by Shao and Zhang \cite[Proof of (5.51)]{SZ}.
Set \[a=\left(1-\dfrac{N-1}{\beta}\right)\dfrac{(n-1)^2 b^2}{\beta^2},\ \sigma^{(1,2)}=S_n-\sigma_1-\sigma_2, V_1=\langle \sigma_1,\sigma^{(1,2)}\rangle^2,\ 
V_2=\langle \sigma_2,\sigma^{(1,2)}\rangle^2,\]
we have 
\[\left|\langle \sigma_1,\sigma^{(1)}\rangle^2-V_1\right|\le Cn,\ \left|\langle \sigma_1,\sigma^{(2)}\rangle^2-V_2\right|\le Cn.\]
It follows that
\begin{align}\label{eq:R4}
\begin{split}
ER_{4}^2&=nE\left(\langle \sigma_1,\sigma^{(1)}\rangle^2-a\right)^2-n(n-1)E\left(\langle \sigma_1,\sigma^{(1)}\rangle^2-a\right)\left(\langle \sigma_2,\sigma^{(2)}\rangle^2-a\right)\\
& \le Cn^5+n(n-1)\left|E\left(\langle \sigma_1,\sigma^{(1)}\rangle^2-V_1+V_1-a\right)\left(\langle \sigma_2,\sigma^{(2)}\rangle^2-V_2+V_2-a\right)\right|\\
& \le Cn^5+n(n-1)|E\left(V_1-a\right)\left(V_2-a\right)|\\
& \le Cn^5+n(n-1)|E\left(V_1-E(V_1|(\sigma_j)_{j>2})\right)\left(V_2-E(V_2|(\sigma_j)_{j>2})\right)|\\
& \quad +n(n-1)|E\left(E(V_1|(\sigma_j)_{j>2})-a\right)\left(E(V_2|(\sigma_j)_{j>2})-a\right)|\\
&:=Cn^5+n(n-1)(|R_{41}|+|R_{42}|).
\end{split}
\end{align}
Define a probability density function 
\begin{align} \label{eq:p12}
p_{12}(x,y)=\dfrac{1}{Z_{12}^2}\exp\left(\dfrac{\beta}{n}\langle x+y,\sigma^{(1,2)}\rangle\right), x,y\in {\mathbb{S}}^{N-1},
\end{align}
where $Z_{12}^2$ is the normalizing constant. Let $(\xi_1,\xi_2)\sim p_{12}(x,y)$ given $(\sigma_j)_{j>2}$, and for $i=1,2$
\begin{align}\label{eq:V}
{\tilde{V}}_i=E\left(\langle \xi_i,\sigma^{(1,2)}\rangle^2|(\sigma_j)_{j>2}\right).
\end{align}
Similar to Shao and Zhang \cite[pages 97, 98]{SZ},
we can show that
\begin{equation}\label{eq:R41a}
\begin{split}
R_{41}&=E\left(\langle \xi_i,\sigma^{(1,2)}\rangle^2-{\tilde{V}}_1\right)\left(\langle \xi_i,\sigma^{(1,2)}\rangle^2-{\tilde{V}}_2\right)+H_1,
\end{split}
\end{equation}
and
\begin{equation}\label{eq:R42a}
\begin{split}
R_{42}&=E\left({\tilde{V}}_1-a\right)\left({\tilde{V}}_2-a\right)+H_2,
\end{split}
\end{equation}
where $|H_1|\le Cn^3$ and $|H_2|\le Cn^3$.
Let 
\begin{align}\label{eq:B12}
b_{12}=\dfrac{\beta |\sigma^{(1,2)}|}{n}.
\end{align}
By \lemref{le:boundV} and the definition of $a$, we have
\begin{align*}
\begin{split}
&\left| {\tilde{V}}_1-a\right|=\left|\left(1 - \dfrac{(N-1)f(b_{12})}{b_{12}} \right)|\sigma^{(1,2)}|^2-\left(1-\dfrac{N-1}{\beta} \right)\dfrac{(n-1)^2b^2}{\beta^2}\right|\\
&\quad =\left|\left(1-\dfrac{N-1}{\beta} \right)\left(|\sigma^{(1,2)}|^2-\dfrac{(n-1)^2b^2}{\beta^2}\right)+(N-1)\left(\dfrac{1}{\beta}-\dfrac{f(b_{12})}{b_{12}}\right)|\sigma^{(1,2)}|^2\right|\\
&\quad \le Cn^2\left(\left|\dfrac{\beta^2|\sigma^{(1,2)}|^2}{n^2} -\dfrac{(n-1)^2b^2}{n^2}\right|+\left|\dfrac{f(b)}{b}-\dfrac{f(b_{12})}{b_{12}}\right|\right)\\
&\quad \le Cn^2\left(\left|\dfrac{\beta^2|S_n|^2}{n^2} -b^2\right|+\left|b_{12}-b\right|\right)+Cn\\
&\quad \le Cn^2\left(\left|\dfrac{\beta|S_n|}{n} -b\right|+\left|\dfrac{\beta|\sigma^{(1,2)}|}{n}-b\right|\right)+Cn
\le Cn^2\left|\dfrac{\beta|S_n|}{n} -b\right|+Cn.
\end{split}
\end{align*}
Using similar estimate for $\left| {\tilde{V}}_2-a\right|$, then we have
\begin{align}\label{eq:R42c}
\begin{split}
E\left|\left({\tilde{V}}_1-a\right)\left({\tilde{V}}_2-a\right)\right|&\le C\left(n^4E\left|\dfrac{\beta|S_n|}{n} -b\right|^2+n^3E\left|\dfrac{\beta|S_n|}{n} -b\right|+n^2\right)\\
&\le Cn^3\mbox{ (by \lemref{le:lema1})}.
\end{split}
\end{align}
Note that given $(\sigma_j)_{j>2}$, $\xi_1$ and $\xi_2$ are conditionally independent.  It implies that
\begin{equation}\label{eq:R41b}
\begin{split}
E\left(\langle \xi_i,\sigma^{(1,2)}\rangle^2-{\tilde{V}}_1\right)\left(\langle \xi_i,\sigma^{(1,2)}\rangle^2-{\tilde{V}}_2\right)=0.
\end{split}
\end{equation}
Combining \eqref{eq:R4}-\eqref{eq:R41b}, we have $ER_{4}^2\le Cn^5$, and so
\begin{equation}\label{eq:R4final}
\begin{split}
E|R_4|\le Cn^{5/2}.
\end{split}
\end{equation}
Combining \eqref{eq:BB}, \eqref{eq:R2final}, \eqref{eq:R3final}, \eqref{eq:R5final} and \eqref{eq:R4final}, we have
\begin{align*} 
E\left|\dfrac{1}{2\lambda}E((W_n-W_{n}^{'})^2|W_n))-B^2 \right| \le C n^{-1/2}.
\end{align*}
The proposition is proved.
\end{proof}

\appendix
\section{Appendix}\label{sec:Appendix}

In this Section, we will prove the technical results that used in the proof of \thmref{th:app2}.
\begin{lem}\label{le:lema1} We have
\[E\Big|\dfrac{\beta|S_n|}{n} - b\Big|^2 \le \dfrac{C}{n}.\]
\end{lem}
\begin{proof}
By the large deviation for $S_n/n$ \cite[Proposition 2]{KN} and the argument in \cite[p. 1126]{KN}, one can prove that
there exists $\varepsilon>0$ such that
\[P\left( \left| \dfrac{\beta|S_n|}{n} - b\right| \ge x \right) \le e^{-Cn {x}^2}\]
for all $0\le x\le \varepsilon$. Since 
$\left|\dfrac{\beta|S_n|}{n} - b\right|\le C,$
it implies that
\begin{align*} 
\begin{split}
E\left|\dfrac{\beta|S_n|}{n} - b\right|^2 &\le 2\int_0^{\varepsilon} x P\left(\left|\dfrac{\beta|S_n|}{n} - b\right|> x\right) dx\\
&\quad +E\left(\left|\dfrac{\beta|S_n|}{n} - b\right|^2 I\left(\left|\dfrac{\beta|S_n|}{n} - b\right|> \varepsilon\right)\right)\\
&\le 2 \int_0^{\varepsilon}xe^{-Cnx^2}dx + C P\left(\left|\dfrac{\beta|S_n|}{n} - b\right|> \varepsilon\right)\\
&\le \dfrac{C}{n} + C e^{-C n\varepsilon^2}
\le \dfrac{C}{n}.
\end{split}
\end{align*}
\end{proof}
\begin{lem}\label{le:boundf}
Let $x>0$ and $f(x) = \dfrac{I_{N/2}(x)}{I_{N/2-1}(x)}$. Then
the following statements hold:
\begin{itemize}
\item[(i)] $0 < f'(x) < \dfrac{1}{N-1}\le 1.$
\item[(ii)] $|(xf(x))''|< 6.$
\item[(iii)] $-5\le \dfrac{-5}{N-1}< \left(\dfrac{f(x)}{x}\right)'<0.$
\end{itemize}
\end{lem}

\begin{proof}
As was showed in \cite[p. 1134]{KN}, we have 
\begin{align} \label{eq:ine01}
f'(x) = 1 - \dfrac{N-1}{x}f(x) - f^2(x).
\end{align}
It implies
\begin{align} \label{eq:ine01a}
\dfrac{f(x)}{x}=\dfrac{1-f'(x)-f^2(x)}{N-1},
\end{align}
and
\begin{align} \label{eq:ine01b}
f^2(x) = 1 - \dfrac{N-1}{x}f(x) - f'(x).
\end{align}
Amos \cite[p. 243]{Amos} proved that
\begin{align} \label{eq:ine02}
0 < f'(x) < \dfrac{f(x)}{x}.
\end{align}
Combining \eqref{eq:ine01a}-\eqref{eq:ine02}, we have
\begin{align} \label{eq:ine03a}
0< f'(x) < \dfrac{f(x)}{x}<\dfrac{1}{N-1},
\mbox{
and
}
f^2(x)<1.
\end{align} 
Therefore,
\begin{align*}
\left|(xf(x))''\right| &= \left|2f'(x) + xf''(x)\right|\\
&= \left|2f'(x) + x \left(-f'(x)\left(\dfrac{N-1}{x}+2f(x) \right)+ \dfrac{N-1}{x^2}f(x) \right) \right|\\
&\le 2 + (N-1)f'(x)+ 2xf(x)f'(x)+\dfrac{(N-1)f(x)}{x}\\
&\le 4 + 2f^2(x)\mbox{ (by the first half of \eqref{eq:ine03a})}\\
&\le 6 \mbox{ (by the second half of \eqref{eq:ine03a})}.
\end{align*}
The proof of (i) and (ii) is completed. For (iii), we have 
\begin{align}\label{eq:ine05}
\left(\dfrac{f(x)}{x}\right)' &= \dfrac{1}{x}\left(f'(x) -\dfrac{f(x)}{x}\right).
\end{align}
Combining the first half of \eqref{eq:ine03a} and \eqref{eq:ine05}, we have $\left(\dfrac{f(x)}{x}\right)' < 0$.
It follows from  \eqref{eq:ine01}, \eqref{eq:ine03a} and \eqref{eq:ine05} that
\begin{align}\label{eq:ine07}
\left(\dfrac{f(x)}{x}\right)' &=\dfrac{1}{x}\left(1 -\dfrac{Nf(x)}{x}\right)-\dfrac{f^2(x)}{x}>\dfrac{1}{x}\left(1 -\dfrac{Nf(x)}{x}\right)-\dfrac{1}{N-1}.
\end{align}
Apply Theorem 2 (a) of N\.{a}sell \cite{Na}, we can show that
\begin{align}\label{eq:ine09}
\dfrac{1}{x}\left(1-\dfrac{Nf(x)}{x}\right)>\dfrac{-4}{N-1}.
\end{align}
Combining \eqref{eq:ine07} and \eqref{eq:ine09}, we have
$\left(\dfrac{f(x)}{x}\right)' >\dfrac{-5}{N-1}.$
The proof of (iii) is completed.
\end{proof}

\begin{lem} \label{le:boundV}
With the notation in the proof of \thmref{th:app2}, we have
\[\tilde{V}_i= |\sigma^{(1,2)}|^2\left(1 - \dfrac{(N-1)f(b_{12})}{b_{12}} \right),\ i=1,2.\]

\end{lem}
\begin{proof}
Let $A_N=2\pi^{N/2}/\Gamma(N/2)$ the Lebesgue measure of $S^{N-1}$. It follows from \eqref{eq:p12} that
\begin{align*}
Z_{12}^2&=\int_{S^{N-1}}\int_{S^{N-1}}\exp\left(\dfrac{\beta}{n}\langle x+y,\sigma^{(1,2)} \rangle\right)d\mu(x)d\mu(y)\\
&=\left(\int_{S^{N-1}}\exp\left(\dfrac{\beta}{n}\langle x,\sigma^{(1,2)} \rangle\right)d\mu(x)\right)^2\\ 
&=\left(\dfrac{A_{N-1}}{A_N}\int_0^{\pi}e^{b_{12}\cos \varphi_{N-2}} \sin^{N-2}\varphi_{N-2}d\varphi_{N-2}\right)^2\\
&=\left(\dfrac{A_{N-1}}{A_N} \dfrac{\sqrt{\pi} \Gamma(N/2-1/2)}{\left(b_{12}/2\right)^{N/2-1}}I_{N/2-1}(b_{12})\right)^2,
\end{align*}
where we have used formula
\[I_{\nu}(z)=\dfrac{1}{\sqrt{\pi}\Gamma(\nu+1/2)}\left(\dfrac{\nu}{2}\right)^{\nu}\int_{0}^{\pi}\exp(z\cos \theta)\sin^{2\nu}\theta d\theta\]
(see, e.g., Exercise 11.5.4 in \cite{AW}) in the last equation. For $i=1,2$, we have
\begin{align*}
\tilde{V}_i &= \dfrac{1}{Z_{12}}\int_{\mathbb{S}^{N-1}} \langle\theta, \sigma^{(1,2)}\rangle^2 \exp\left[ \dfrac{\beta}{n} \langle \theta, \sigma^{(1,2)}\rangle\right]d\mu(\theta)\\
&= \dfrac{1}{Z_{12}}\int_{\mathbb{S}^{N-1}} |\sigma^{(1,2)}|^2\left\langle\theta, \dfrac{\sigma^{(1,2)}}{|\sigma^{(1,2)}|}\right\rangle^2 \exp\left( \dfrac{\beta|\sigma^{(1,2)}|}{n} \left\langle\theta, \dfrac{\sigma^{(1,2)}}{|\sigma^{(1,2)}|}\right\rangle\right)d\mu(\theta) \\
&= |\sigma^{(1,2)}|^2\dfrac{A_{N-1}}{A_N Z_{12}}\int_0^{\pi} \cos^2\varphi_{N-2}\sin^{N-2}\varphi_{N-2}e^{b_{12}\cos\varphi_{N-2}} d\varphi_{N-2}\\
&= |\sigma^{(1,2)}|^2\dfrac{A_{N-1}}{A_N Z_{12}}\int_0^{\pi}e^{b_{12} \cos\varphi_{N-2}}\sin^{N-2}\varphi_{N-2} d\varphi_{N-2}\\
&\quad -\int_{0}^{\pi}e^{b_{12}\cos\varphi_{N-2}}\sin^{N}\varphi_{N-2} d\varphi_{N-2}\\
&= \left(1-\dfrac{A_{N-1}}{A_N Z_{12}}\int_{0}^{\pi}e^{b_{12}\cos\varphi_{N-2}}\sin^{N}\varphi_{N-2} d\varphi_{N-2} \right)|\sigma^{(1,2)}|^2\\
&= \left(1-\dfrac{A_{N-1}}{A_N Z_{12}}\dfrac{\sqrt{\pi} \Gamma(N/2+1/2)}{\left(b_{12}/2\right)^{N/2}}I_{N/2}(b_{12}) \right)|\sigma^{(1,2)}|^2\\
&= \left(1 - \dfrac{(N-1)f(b_{12})}{b_{12}} \right)|\sigma^{(1,2)}|^2.
\end{align*}
\end{proof}	
Finally, we would like to note again that \propref{pr:EL} is a special case of
Theorem 2.4 of Eichelsbacher and L\"{o}we \cite{EL} or Theorem 13.1 in \cite{CGS}, but the constants
in the bound may be different from those of Theorem 2.4 in \cite{EL} or Theorem 13.1 in \cite{CGS}. Since the proof is short and simple, we will present here.
\begin{proof}[Proof of the \propref{pr:EL}]
Let $h: \mathbb{R}\to\mathbb{R}$ such that $\|h'\|\le 1$ and $E|h(Z)|<\infty$, and 
let $f:=f_h$ be the unique solution to the Stein's equation
$f'(w)-wf(w)=h(w)-Eh(Z).$
Since $(W,W^{'})$ is an exchangeable pair and $E(W-W'|W) = \lambda(W + R)$,
\begin{align*}
 0&=E(W-W')(f(W)+f(W'))\\
&=E(W-W')(f(W')-f(W))+2Ef(W)(W-W')\\
&=E(W-W')(f(W')-f(W))+2\lambda Ef(W)E(W-W'|W)\\
&=E\Delta(f(W')-f(W))+2\lambda EWf(W)+2\lambda Ef(W)R.
\end{align*}
It thus follows that
\begin{align}\label{eq:Kanto}
\begin{split}
&|Eh(W)-Eh(Z)|\\
&=|E(f'(W)-Wf(W))|\\
&=\left|E\left( f'(W)+\dfrac{1}{2\lambda}E\Delta(f(W')-f(W))+Ef(W)R\right)\right|\\
&= \left|E\left(f'(W)\left(1-\dfrac{1}{2\lambda}E(\Delta^2|W)\right)+\dfrac{1}{2\lambda}\Delta\left(f(W')-f(W)+\Delta f'(W)\right)+f(W)R\right)\right|\\
&\le \|f'\| E\left|1-\dfrac{1}{2\lambda}E(\Delta^2|W)\right|+\dfrac{1}{4\lambda}\|f''\|E\left|\Delta\right|^3+\|f\|E|R|.
\end{split}
\end{align}
By Lemma 2.4 in \cite{CGS} we have
\begin{align}\label{eq:boundCGS}
\|f\|\le 2,\ \|f'\|\le \sqrt{2/\pi},\ \|f''\|\le 2.
\end{align}
The conclusion of the proposition follows from \eqref{eq:Kanto} and \eqref{eq:boundCGS}.
\end{proof}

\ACKNO{We are grateful to the Associate Editor and an anonymous referee whose comments have been very
useful to make our development clearer and more comprehensive.}

\end{document}